\documentclass[12pt,a4paper]{amsart}
\usepackage[top=35mm, bottom=30mm, left=30mm, right=30mm]{geometry}
\usepackage{mathptmx}
\usepackage{mathrsfs}

\usepackage{verbatim}
\usepackage{url}
\usepackage[all]{xy}
\usepackage{xcolor}
\usepackage[colorlinks=true,citecolor=blue]{hyperref}

\usepackage{amsmath,amsthm}
\usepackage{amssymb}

\usepackage{enumerate}

\usepackage{graphicx}

\newtheorem{thm}{Theorem}[section]
\newtheorem{cor}[thm]{Corollary}
\newtheorem{lem}[thm]{Lemma}

\theoremstyle{definition}

\numberwithin{equation}{section}

\begin{document}


\baselineskip=17pt


\title{Continua having distal minimal actions by amenable groups}

\author{Enhui Shi}

\address[E.H. Shi]{School of Mathematical Sciences, Soochow University, Suzhou 215006, P. R. China}
\email{ehshi@suda.edu.cn}

\begin{abstract}
Let $X$ be a non-degenerate connected compact metric space. If $X$ admits a distal minimal action by a finitely generated amenable group,
then the first \v Cech cohomology group $ {\check H}^1(X)$ with integer coefficients is nontrivial.
In particular, if $X$ is homotopically equivalent to a CW complex, then $X$ cannot be simply connected.
\end{abstract}


\maketitle

\pagestyle{myheadings} \markboth{E. H. Shi }{Distal minimal actions by amenable groups}

\section{Introduction}

The notion of distality was introduced by Hilbert for better understanding equicontinuity (\cite{El}). The study of minimal distal systems culminates in the
 beautiful structure theorem of H. Furstenberg (\cite{Fu}), which describes completely the relations between distality and equicontinuity for minimal systems.
 An interesting question is what compact manifold can support a  distal minimal group action? Clearly, the answer to this question depends on the
 topology of the phase space and the algebraic structure of the acting group. A remarkable result says that if a nontrivial space $X$ admits a
 distal minimal actions by abelian groups, then $X$ cannot be simply connected (see e.g. \cite[Chapter 7-Theorem 16]{Au}). Thus the $n$-sphere $\mathbb S^n$
does not admit any distal minimal abelain group actions. In \cite{Sh}, the author showed that if $X$ is a closed surface and $\Gamma$ is a lattice in ${\rm SL}(n, \mathbb R)$ with $n\geq 3$, then $\Gamma$ cannot act on $X$ distally and minimally.

We consider amenable group actions on continua and get the following theorem.

\begin{thm}\label{main theorem}
Let $X$ be a non-degenerate compact connected metric space. If $X$ admits a distal minimal action by a finitely generated amenable group,
then the first \v Cech cohomology group ${\check H}^1(X)$ with integer coefficients is nontrivial. In particular,
if $X$ is homotopically equivalent to a CW complex, then $X$ cannot be simply connected.
\end{thm}

The following corollary is immediate.

\begin{cor}\label{main corollary}
The $n$-sphere $\mathbb S^n$ ($n\geq 2$)
does not admit any distal minimal actions by finitely generated amenable groups.
\end{cor}

Here we remark that the class of amenable groups is strictly larger than that of abelian groups, which contains all solvable groups.
In addition, there do exist  distal minimal  actions on $\mathbb S^n$ by nonamenable groups, such the actions generated by some irrational rigid
 rotations around different axes of $\mathbb R^{n+1}$.

\section{Preliminaries}

In this section, we will recall some basic notions and introduce some results which will be used in the proof of the main theorem.

\subsection{Distal group actions}

Let $X$ be a topological space and let ${\rm Homeo}(X)$ be
the homeomorphism group of $X$. Suppose $G$ is a topological group. A group
homomorphism $\phi: G\rightarrow {\rm Homeo}(X)$ is called a {\it continuous
action} of $G$ on $X$ if $(x, g)\mapsto \phi(g)(x)$ is continuous; we use the symbol $(X, G, \phi)$ to denote this action.
The action $\phi$ is said to be {\it faithful} if it is injective.  For brevity, we usually use $gx$ or $g(x)$ instead of $\phi(g)(x)$
and use $(X, G)$ instead of $(X, G, \phi)$ if no confusion occurs.

For $x\in X$,  the {\it orbit} of $x$  is the set $Gx\equiv\{gx:g\in
G\}$; $K\subset X$ is called {\it $G$ invariant} if $Gx\in K $ for every $x\in X$;
$(X, G, \phi)$ is called {\it minimal} if $Gx$ is dense in $X$ for every $x\in X$, which is equivalent
to that $G$ has no proper closed invariant set; is called {\it transitive} if $Gx=X$ for every $x\in X$.
If $K\subset X$ is $G$ invariant, then we naturally get a restriction action $\phi|K$ of $G$ on $X$;
if $K$ is closed and nonempty, and the restriction action $(K, G, \phi|K)$ is minimal,
then we call $K$ a {\it minimal set}  of $X$ or of the action. It is well known that $(X, G, \phi)$
always has a minimal set when $X$ is a compact metric space.

Suppose $(X, G, \phi)$ and  $(Y, G, \psi)$ are two actions. If there is a continuous surjection  $f:X\rightarrow Y$ such that
$f(\phi(g)x)=\psi(g)f(x)$ for every $g\in G$ and every $x\in X$, then we say $f$ is a {\it homomorphism} and $(Y, G, \psi)$
is a {\it factor} of $(X, G, \phi)$. If $Y$ is a single point, then we call $(Y, G, \psi)$  a {\it trivial factor} of  $(X, G, \phi)$.

Assume further that $X$ is a compact metric space with metric $d$. The action $(X, G, \phi)$ is called
{\it equicontinuous} if for every $\epsilon>0$ there is a $\delta>0$ such that $d(gx, gy)<\epsilon$ whenever $d(x, y)<\delta$;
is called {\it distal}, if for every $x\not=y\in X$, $\inf_{g\in G}d(gx, gy)>0$. Clearly, equicontinuity implies distality.

The following results can be found in \cite{Au}.

\begin{thm}[\cite{Au}, p.98]\label{homomorphism open}
Let $(X, G, \phi)$ and  $(Y, G, \psi)$ be distal minimal actions, and let $f:X\rightarrow Y$ be a homomorphism.
Then $f$ is open.
\end{thm}

\begin{thm}[\cite{Au}, p.104]\label{maximal fator}
Suppose $X$ is not a single point. If $(X, G, \phi)$ is distal minimal, then it has a nontrivial equicontinuous factor.
\end{thm}

\begin{thm}[\cite{Au}, p.52]\label{compact group}
Let $(X, G, \phi)$ be  equicontinuous. Then the closure $\overline{\phi(G)}$ in $C(X,X)$ with respect to
the uniform convergence topology is a compact topological group.
\end{thm}

\subsection{Amenable groups}

{\it Amenability} was first introduced by von Neumann. Recall that a
countable group $G$ is {\it amenable} if there is a
sequence of finite sets $F_i$ ($i=1, 2, 3,\ \dots$) such that
$\lim\limits_{i\to\infty}\frac{|gF_i\bigtriangleup F_i|}{|F_i|}=0$
for every $g\in G$, where $|F_i|$ is the number of elements in
$F_i$; the set $F_i$ is called a {\text{F{\o}lner} \it set}.
For an abstract group $G$, if there is a sequence of normal subgroups
$G=G_0\rhd G_1\rhd...\rhd G_n=\{e\}$ such that $G_i/G_{i+1}$ is commutative for each
$i$, then $G$ is called {\it solvable}.

Now we list some well known facts about amenable groups and solvable groups. One may consult \cite{Pa}
for the details.

\begin{thm}\label{amenable}
(1)  Solvable groups and finite groups are amenable; (2) any
group containing a free noncommutative subgroup cannot be amenable; (3) every
subgroup of an amenable group (resp. solvable group) is amenable (resp. solvable);
(4) every quotient group of an amenable group (resp. solvable group) is amenable (resp. solvable).
\end{thm}

The following remarkable result is known as Tits Alternative (see \cite{Ti}).

\begin{thm}\label{Tits}
Let $\Gamma$ be a finitely generated subgroup of a linear group. Then either $\Gamma$ contains a free nonabelian subgroup,
or $\Gamma$ has a finite index solvable subgroup.
\end{thm}

\subsection{Compact Lie groups}

Let $G$ be a connected Lie group and let ${\text {Lie}}(G)$ be the Lie algebra of $G$.
Recall that $G$ is said to be {\it solvable} if its Lie algebra is solvable, that is
there is a sequence of ideals ${\text {Lie}}(G)=\Im_0\rhd \Im_1\rhd...\rhd \Im_n=\{0\}$
such that $\Im_i/\Im_{i+1}$ is commutative for each $i$; this is equivalent to
the existence of a sequence of closed normal subgroups $G=G_0\rhd G_1\rhd...\rhd G_n=0$
such that $G_i/G_{i+1}$ is commutative for each $i$. If the Lie algebra ${\text {Lie}}(G)$
of $G$ contains no nontrivial solvable ideal, then $G$ is said to be {\it semisimple}.

The following theorems are classical in the theory of Lie groups.

\begin{thm}[\cite{Kn}, Corollary 4.25 ]\label{compact Lie group}
Let $G$ be a compact Lie group and let $\Im$ be the Lie algebra of $G$. Then
$\Im=Z(\Im)\bigoplus [\Im, \Im]$, where $Z(\Im)$ is the center of $\Im$ and
$[\Im, \Im]$ is semisimple.
\end{thm}

\begin{thm}[\cite{Kn}, Corollary 1.103]\label{abelian compact}
Let $G$ be a compact connected commutative Lie group of dimension $n$. Then $G$
is isomorphic to the $n$-torus $\mathbb T^n$.
\end{thm}

\begin{cor}\label{compact solvable}
Let $G$ be a connected compact Lie group. If $G$ is solvable, then $G$ is isomorphic to the $n$-torus $\mathbb T^n$.
\end{cor}

\begin{proof}
Let $\Im$ be the Lie algebra of $G$ and let $Z(\Im)$ be its center. If $[\Im, \Im]\not=0$,
then $\Im/Z(\Im)$ is semisimple by Theorem \ref{compact Lie group}.
However, this is impossible since $\Im/Z(\Im)$ is also solvable. So $[\Im, \Im]=0$ and hence
$\Im$ is commutative. This implies $G$ is commutative, since $G$ is connected.
It follows from Theorem \ref{abelian compact} that $G$ is isomorphic to the $n$-torus $\mathbb T^n$,
where $n$ is the dimension of $G$.
\end{proof}

\begin{thm}[\cite{Kn}, Corollary 4.22]\label{finite representation}
Let $G$ be a compact Lie group. Then $G$ is isomorphic to a closed linear group.
\end{thm}

\subsection{Compact transformation groups}

Let $(X, G, \phi)$ be a group action and $H$ be a closed subgroup of $G$. Then we use $X/H$ to denote the orbit space
under the $H$ action, which is endowed with the quotient space topology. We use $G/H$ to denote the coset space with the quotient topology,
which is also the orbit space obtained by the left translation action on $G$ by $H$. If $H$ is a normal closed subgroup of $G$, then
$G/H$ is a topological group.

The following theorems can be seen in \cite{MZ}. We only state them in some special cases which are enough for our uses.

\begin{thm}[\cite{MZ}, p.65]\label{homogeneous space}
Let $X$ be a compact metric space and let $(X, G)$ be an action of group $G$ on $X$. Suppose $G$ is compact. Then for every
$x\in X$, $G/G_x$ is homeomorphic to $Gx$, where $G_x=\{g\in G: gx=x\}$.
\end{thm}

\begin{thm}[\cite{MZ}, p.99]\label{small group} Let $G$ be a compact group and let $U$ be an open neighborhood of the identity $e$.
Then $U$ contains a normal subgroup $H$ of $G$ such that $G/H$ is isomorphic to a Lie group.
\end{thm}

\begin{thm}[\cite{MZ}, p.61]\label{induce action} Let $X$ be a compact metric space and let $(X, G)$ be an action of group $G$ on $X$.
 Suppose $G$ is compact and $H$ is a closed normal subgroup of  $G$. Then  $G/H$ can
 act on $X/H$ by letting $gH.H(x)=H(gx)$ for $gH\in G/H$ and $H(x)\in X/H$.
\end{thm}

\subsection{First {\v C}ech cohomology group}

First we will recall an equivalent definition of the first {\v C}ech cohomology group with integer coefficients.
Let $\mathbb S^1$ be the unit circle in the complex plane. For any paracompact normal  space $X$,
let $C(X, \mathbb S^1)$ be the set of all continuous functions from $X$ to $\mathbb S^1$, and
let $I(X, \mathbb S^1)$ be the set of all $f\in C(X, \mathbb S^1)$ which is inessential (i.e. $f$ is homotopoic to a constant map
from $X$ to $\mathbb S^1$). Then under pointwise complex multiplication, $C(X, \mathbb S^1)$ becomes a commutative group
and  $I(X, \mathbb S^1)$ is a subgroup of  $C(X, \mathbb S^1)$. Define the first cohomology group ${\check H}^1(X)$ of $X$ by
${\check H}^1(X)=C(X, \mathbb S^1)/I(X, \mathbb S^1).$

Suppose $f:X\rightarrow Y$ is continuous. Then  $f$ naturally induced a group homomorphism
$f^*:H^1(Y)\rightarrow H^1(X)$ by letting $f^*([g])=[g\circ f]$ for any $[g]\in H^1(Y)$.
 The map $f$ is called {\it confluent} if for any subcontinuum $B$ of $Y$
and any component $A$ of $f^{-1}(B)$, we have $f(A)=B$.

The following theorem is due to Lelek (see \cite{Le} or \cite[Theorem 13.45]{Nad}).

\begin{thm}\label{homology injective}
Let $f:X\rightarrow Y$ be a confluent map from continuum $X$ onto continuum $Y$. Then
$f^*:H^1(Y)\rightarrow H^1(X)$ is injective.
\end{thm}

The following theorem is due to Whyburn (see \cite{Wh2} or \cite[Theorem 13.14]{Nad}).

\begin{thm}\label{open confluent}
Every open map of one compact metric space onto another is confluent.
\end{thm}

From Theorem \ref{homology injective} and Theorem \ref{open confluent}, we immediately get the following corollary.

\begin{cor}\label{open inverse}
Let $f:X\rightarrow Y$ be an open map from continuum $X$ onto continuum $Y$. Then
$f^*:H^1(Y)\rightarrow H^1(X)$ is injective.
\end{cor}

\section{Proof of the main theorem}

\begin{lem}\label{solvable closure}
Let $G$ be a Lie group and $\Gamma$ be a dense subgroup of $G$. If $\Gamma$ is solvable as an abstract group,
then $G$ is a solvable Lie group.
\end{lem}

\begin{proof}
Since $\Gamma$ is solvable, we have a sequence of normal subgroups 
$\Gamma=\Gamma_0\rhd \Gamma_1\rhd...\rhd \Gamma_n=\{e\}$ such that $\Gamma_i/\Gamma_{i+1}$ is commutative for each
$i$. Let $G_i=\overline \Gamma_i$. Then we get a decreasing sequence of closed normal subgroups
$G=G_0\rhd G_1\rhd...\rhd G_n=\{e\}$. Since $\Gamma_{i+1}\subset G_{i+1}$ and $\Gamma_i/\Gamma_{i+1}$ is commutative,
$\Gamma_i/(\Gamma_i\cap G_{i+1})$ is commutative. So $G_i/G_{i+1}$ is commutative by the continuity of group operations.
Then $G$ is solvable.
\end{proof}

\begin{lem}\label{component group}
Let $G$ be a compact Lie group and $G_0$ be the connected component of $e\in G$. Suppose $G$ acts transitively on a connected
compact manifold $M$. Then the $G_0$ action on $M$ is also transitive.
\end{lem}

\begin{proof}
It is well known that $G_0$ is a clopen normal subgroup of $G$. Since $G$ is compact, $G_0$ has finite index in $G$.
Let $G=g_1G_0\cup...\cup g_kG_0$ be the coset decomposition, where $k=[G:G_0]$. Fix an $x_0\in M$.
Then $\cup_{i=1}^kg_iG_0x_0=Gx_0=M$, since the $G$ action is transitive. So, $G_0x_0$ contains a nonempty open set,
which implies that $G_0x_0$ is open by the homogeneous of the orbit $G_0x_0$. Thus $G_0x_0$ is clopen in $M$.
Hence $G_0x_0=M$ by the connectedness of $M$.
\end{proof}

\noindent {\bf Proof of Theorem \ref{main theorem}.} Let $X$ be a connected compact metric space and let $\Gamma$ be a finitely generated amenable group.
Suppose $X$ admits a minimal distal action $\phi:\Gamma\rightarrow $Homeo$(X)$. We will show that
the first cohomology group ${\check H}^1(X)$ with integer coefficients is nontrivial.

By Theorem \ref{maximal fator}, there is a nontrivial equicontinuous factor $(Y, \Gamma, \psi)$ of $(X, \Gamma, \phi)$,
which is still minimal and connected. Set $H=\overline{\psi(\Gamma)}$. From Theorem \ref{compact group},  $H$ is a compact subgroup of ${\rm Homeo}(Y)$ with respect
to the uniform convergence topology.  Applying Theorem \ref{small group}, we can take a small normal subgroup $H'$ of $H$ such that $H/H'$ is a Lie group and $H'y$ is a proper subset of $Y$ for every $y\in Y$.

Then we get an equicontinuous
action $\psi'$ of the Lie group $H/H'$ on the quotient space $Y/H'$ by Theorem \ref{induce action}; in particular, $Y/H'$ is homeomorphic
 to the quotient space $(H/H')/F_y$ by Theorem \ref{homogeneous space}, where $y\in Y$ and $F_y=\{gH'\in H/H':gH'.H'y=H'y\}$. 
 This implies that $Y/H'$ is a connected compact manifold of dimension $\geq 1$ (see \cite[Theorem 3.58]{War}), and hence
 $H/H'$ is a compact Lie group of dimension $\geq 1$. From Theorem \ref{amenable}-(4),
$\psi(\Gamma)H'$ is an amenable subgroup of $H/H'$ (as abstract groups); from Theorem \ref{finite representation}, Theorem \ref{Tits}, and Theorem \ref{amenable}-(2), we
see that $\psi(\Gamma)H'$ has a solvable subgroup $\Gamma'$ of finite index. Then $\overline {\Gamma'}$
is a closed subgroup of $H/H'$ with finite index, and hence contains the connected component $(H/H')_0$ of $H/H'$. Since $\Gamma'$
is solvable as a abstract group, $\overline {\Gamma'}$ is a solvable Lie group by Lemma \ref{solvable closure}. So,
$(H/H')_0$ is a connected compact solvable Lie group, which is then isomorphic to $\mathbb T^n$ with $n\geq 1$ by Corollary \ref{compact solvable}.

It follows from Lemma \ref{component group} that the $(H/H')_0$ action on $Y/H'$ is still transitive. So, $Y/H'$
is homeomorphic to $(H/H')_0/(F_y\cap (H/H')_0)$. Since $(H/H')_0$
is isomorphic to $\mathbb T^n$, $Y/H'$ is homeomorphic to $\mathbb T^m$ for some $1\leq m\leq n$.
Thus the first {\v C}ech cohomology group ${\check H}^1(Y/H')\cong {\check H}^1(\mathbb T^m)\cong \mathbb Z^m\not=0$.
 Noting that $(Y/H', \Gamma)$ is a minimal equicontinuous factor of $(X, \Gamma)$,
we denote by $\pi$ the factor map between them. Then $\pi$ is open and surjective by Theorem \ref{homomorphism open}.
Applying Theorem \ref{open inverse}, we have $\pi^*:{\check H}^1(Y/H')\rightarrow {\check H}^1(X)$ is injective;
in particular, ${\check H}^1(X)\not=0$. Since  the first {\v C}ech cohomology group coincides with the first singular cohomology group
when $X$ is homotopically equivalent to a CW complex, the fundamental group of $X$ is nontrivial.
\hfill{$\Box$}



\end{document}